\documentclass{amsart}

\usepackage{a4wide}
\usepackage[dvips]{graphicx}
\usepackage{amssymb}
\usepackage{amsmath}
\usepackage{vaucanson-g}

\newtheorem{proposition}{Proposition}[section]
\newtheorem{theorem}[proposition]{Theorem}
\newtheorem{lemma}[proposition]{Lemma}
\newtheorem{corollary}[proposition]{Corollary}
\theoremstyle{remark} \theoremstyle{definition}
\newtheorem{example}{Example}
\newtheorem{remark}{Remark}
\newtheorem{definition}[proposition]{Definition}

\newcommand {\T}  {{\mathcal T}}

\numberwithin{equation}{section}

\title[Cut sets of IFS]
{On Cut Sets of Attractors of Iterated Function Systems}

\author{Beno\^it Loridant}
\thanks{The authors were supported by the project P22855 of the FWF (Austrian Science Fund) and the project I 1136 of the FWF and the ANR  (French National Research Agency).}
\address[Beno\^it Loridant, J\"org Thuswaldner ]{Montanuniversit\"at Leoben,
    Franz Josefstrasse 18, Leoben 8700, Austria}

\author{Jun Luo}
\address[Jun Luo]{Department of Statistics,
    Sun Yat-Sen University, Guangzhou 512075, China}
    
\author{Tarek Sellami}
\address[Tarek Sellami]{Sfax University, Faculty of sciences of Sfax, Department of mathematics, Route Soukra, BP 802, 3018 Sfax, Tunisia.}

\author{J\"org M. Thuswaldner}


\date{\today}

\begin{document}

\begin{abstract}
In this paper, we study cut sets of attractors of iteration function systems (IFS) in $\mathbb{R}^d$. Under natural conditions, we show that all irreducible cut sets of these attractors are perfect sets or single points. This leads to a criterion for the existence of cut points of IFS attractors. If the IFS attractors are self-affine tiles, our results become algorithmically checkable and can be used to exhibit cut points with the help of Hata graphs. This enables us to construct cut points of some self-affine tiles studied in the literature.
\end{abstract}

\maketitle
\begin{section}{Introduction}
To an iterated function system (IFS) $\{f_i\}_{i=1}^q$ of injective contractions on a complete metric space, there corresponds a unique nonempty compact set $T$ with the \emph{self-similarity} property 
\begin{equation}\label{IFSEQ}
T = \bigcup_{i=1}^q f_i(T).
\end{equation}
This self-similar set is called the \emph{attractor of the IFS} (see~\cite{Hutchinson81}). 

Special instances of IFS are so-called self-affine tiles. Let $A\in\mathbb{R}^{d\times d}$ be an expanding real matrix 
whose determinant $\det A\ne0$ is an integer and a subset ${\mathcal
D}=\left\{e_1,\ldots,e_{|\det A|}\right\}\subset\mathbb{R}^d$. If the attractor $T$
satisfying
\begin{equation}\label{SATEQ}
AT=T+\mathcal{D}:=\bigcup\limits_{e\in{\mathcal D}}T+e 
\end{equation}
has non-empty interior, then it
is called a \emph{self-affine tile}. The associated IFS reads $\{f_e\}_{e\in\mathcal{D}}$, where
$f_e(x)=A^{-1}(x+e)
$
is a contraction for some norm of $\mathbb{R}^d$ (see~\cite{LagariasWang96}). If all the coefficients of $A$ are
integers, ${\mathcal D}\subset{\mathbb Z}^d$ and $T+{\mathbb Z}^d$
is a tiling of ${\mathbb R}^d$, then $T$ is called an \emph{integral self-affine} ${\mathbb
Z}^d$-\emph{tile}.

 There is a vast literature on topological properties of IFS attractors and self-affine tiles. A rather common assumption on the attractor of an IFS $\{f_i\}_{i=1}^q$ is the \emph{open set condition} (OSC). We say that the IFS (or its attractor) satisfies the \emph{open set condition} whenever there exists a bounded open set $U$ with  $\bigcup_{i=1}^q f_i(U)\subset U $ and $f_i(U)\cap f_j(U) = \emptyset$ for all $i \neq j$ (see for example~\cite{Falconer97}). Under this condition, the contracted copies $f_i(T)$ of $T$ do not overlap. Hata showed in~\cite{H} that a connected attractor is even a locally connected continuum. In the case of a plane attractor ($d = 2$), Luo, Rao and Tan proved in~\cite{LRT} the disk-likeness of self-similar connected tiles with connected interior. The connectedness of self-affine tiles was studied in~\cite{KiratLau00}, and Bandt and Wang characterized in~\cite{BW} self-affine plane tiles that are homeomorphic to a disk by the number and location of their neighbors in the tiling. The structure of the interior components of self-similar tiles with disconnected interior is investigated by Ngai and Tang in~\cite{NT}. 
 
 These studies usually rely on the way intersections happen between the natural subdivisions of the attractor obtained by iterating the set equation~(\ref{IFSEQ}). The underlying tools are the so-called \emph{Hata graphs}. Suppose that $T$ is the attractor of an IFS $\{f_1,\ldots,f_q\}$. For $n\in\mathbb{N}$ and any sequence $\alpha=i_1i_2\cdots i_n\in\{1,2,\ldots,q\}^n$, let $f_\alpha: =f_{i_1}\circ f_{i_2}\circ \cdots \circ f_{i_n}$. The
{\em $n$-th Hata Graph} $G_n$ for
the IFS $\left\{f_1,f_2,\ldots,f_q\right\}$ is defined as the graph
with vertex set $\{1,2,\ldots,q\}^n$ so that two elements
$\alpha,\beta\in\{1,2,\ldots,q\}^n$ are incident whenever
$f_\alpha(T)\cap f_\beta(T)\ne\emptyset$.

Hata showed in~\cite{H}  that the attractor is connected if and only if
the first Hata graph $G_1$ is connected. Therefore, sometimes $G_n$ is called the $n$-th {\em connectivity graph} of the IFS. Hata graphs $G_n$ 
provide a combinatoric viewpoint to explore the topology of $T$ and its boundary, such as connectivity, homeomorphy to a simple arc or to a closed disk~\cite{H,KiratLau00,AT}.

In this paper, we are concerned with \emph{cut sets} of self-similar sets. These are subsets $X$ of a connected self-similar set $T$ such that $T\setminus X$ is no longer connected. The study of cut sets is of great importance for the understanding of fractal sets with a wild topology. They were used in~\cite{AkiyamaDorferThuswaldnerWinkler09} to give a combinatorial description of the fundamental group of the Sierpinski gasket, and more generally of one-dimensional spaces~\cite{DorferThuswaldnerWinkler13}. On the opposite, the lack of cut points has an impact on the boundary of the complement of locally connected plane continua~\cite{WhyburnDuda79}. This was exploited in~\cite{NT1} to show the homeomorphy to the closed disk  of interior components of some self-affine tiles. 

We obtain the following results.  In Section~\ref{sec:IrrPerf}, we investigate cut sets of general self-similar sets and  give conditions under which the irreducible cut sets of a connected self-similar set without cut points are perfect sets (Theorem~\ref{general} and its corollary). In Section~\ref{SATIrrPerf} we  formulate this result for the particular case of integral self-affine tiles and obtain that the above conditions become algorithmically checkable. Finally, in Section~\ref{sec:SATCuts}, we present a new application of Hata graphs and give a method to detect  cut points in a self-affine $\mathbb{Z}^d$-tile (Theorem~\ref{cutpoint}). Our results are illustrated throughout the paper by several examples. 

\end{section}

\begin{section}{Irreducible cuts are perfect}\label{sec:IrrPerf}

In this section, the main result concerns cut sets of general self-similar sets. Let us start with fundamental definitions. 

We recall that a  \emph{cut set} $X$ of a connected set $Y$ is a subset of $Y$ such that $Y\setminus
X$ is disconnected. A  cut set $X$ of a connected set $Y$ is \emph{irreducible} if  $Y\setminus X_0$ is connected for any
proper subset $X_0$ of $X$ which is closed in $Y$. The {\em derived set} of a set $X\subset Y$ is the set of limit points, denoted by $X'$. 
Finally, for a finite set $A$, the set $A^*$ denotes the set of all nonempty finite words on $A$. For $k\geq 1$, $A^k$ is then the subset of words of length $k$.

\begin{theorem}\label{general}
Let $T$ be the attractor of an IFS $\{f_1,\ldots,f_q\}$ with
injective contractions. Suppose that $T$ is connected and let $E$ be the
exceptional set defined as
\begin{equation}\label{E}
\displaystyle E:=\bigcup_{\begin{array}{c}\alpha,\beta\in \{1,\ldots,
q\}^*,\\ \#f_\alpha(T)\cap f_\beta(T)=1\end{array}} f_\alpha(T) \cap f_\beta(T).
\end{equation}
Then either  $T$ has a cut point or each irreducible cut set $X$ of $T$ admits a partition
$X=X'\cup X_0$ with $X_0\subset E$.
\end{theorem}

\begin{corollary}
Let $\{f_1,\ldots,f_q\}$ be a set of injective contractions. Assume
that the associated attractor $T$ is connected and satisfies
$\#f_\alpha(T)\cap f_\beta{(T)}\ne1$ for each pair of elements
$\alpha,\beta\in\{1,2,\ldots,q\}^\star$. Then either $T$ has a cut point or each irreducible cut set $X$ of $T$ is a perfect set.
\end{corollary}

\begin{proof}[Proof for Theorem~\ref{general}] Assume that $T$ has no cut point
and let $X$ be an irreducible cut set of $T$. Then for any separation
$T\setminus X=A\cup B$ we have $\overline{A}\cap\overline{B}=X$. Let  $X_0:=X\setminus E$. 
We have to show that $X_0$ contains no isolated
point of $X$.

Assume on the contrary that $x_0\in X_0$ is isolated in $X$. Let
$\varepsilon>0$ be the distance between $x_0$ and
$X\setminus\{x_0\}$. Choose $k\in{\mathbb N}$ large enough that
$\textrm{diam} f_\alpha(T)<\varepsilon$ for each
$\alpha\in\{1,2,\ldots,q\}^k$. Let ${\mathcal W}$ be the collection
of all the indices $\alpha\in\{1,2,\ldots,q\}^k$ with $x_0\in
f_\alpha(T)$ and define
\[A_1:=\bigcup\limits_{\alpha\in{\mathcal W}}\left[f_\alpha(T)\setminus\{x_0\}\right].\]
Observe that $f_\alpha(T)\setminus\{x_0\}$ is connected for each
$\alpha\in{\mathcal W}$. Indeed, if it were disconnected then $x_0$
would be a cut point of $f_\alpha(T)$ and therefore, since
$f_\alpha$ is injective, $f_\alpha^{-1}(x_0)$ is a cut point of $T$,
which contradicts our assumption.

Note that $f_\alpha(T)\cap f_\beta(T)\ne\emptyset$ for
$\alpha,\beta\in{\mathcal W}$. Since $x_0\notin E$ this implies that
$\#\left[f_\alpha(T)\cap f_\beta(T)\right]\ge2$. Hence, $A_1$ is the
union of finitely many connected sets having pairwise nonempty
intersection. Thus $A_1$ is connected and therefore contained either
in $A$ or in $B$. Without loss of generality, assume $A_1\subset A$.
Set a closed set $B_1$ by
\[B_1:=\bigcup\limits_{\alpha\in\{1,2,\ldots,q\}^k\setminus{\mathcal W}} f_\alpha(T).\]
Then $T\setminus\{x_0\}=A_1\cup B_1$. Since $A_1\subset A$ and $B\cap A=\emptyset$,  we have necessarily $B\subset B_1$, and $B_1$ does not contain $x_0$. This means that
$\overline{A}\cap\overline{B}$ is contained in $\overline{A}\cap
B_1$ and therefore $\overline{A}\cap\overline{B}=X$ does not contain
$x_0$. This contradicts the assumption that $x_0\in X$.
\end{proof}

\begin{example}
Let $n \in \mathbb{N}$ be even and define the squares
\[
Q_n(i,j) := \left[\frac{i}{n},\frac{i+1}{n}\right] \times
\left[\frac{j}{n},\frac{j+1}{n}\right] \qquad (i,j \in
\{0,1,\ldots, n-1\}).
\]
\begin{figure}[h]
\includegraphics[width=4cm]{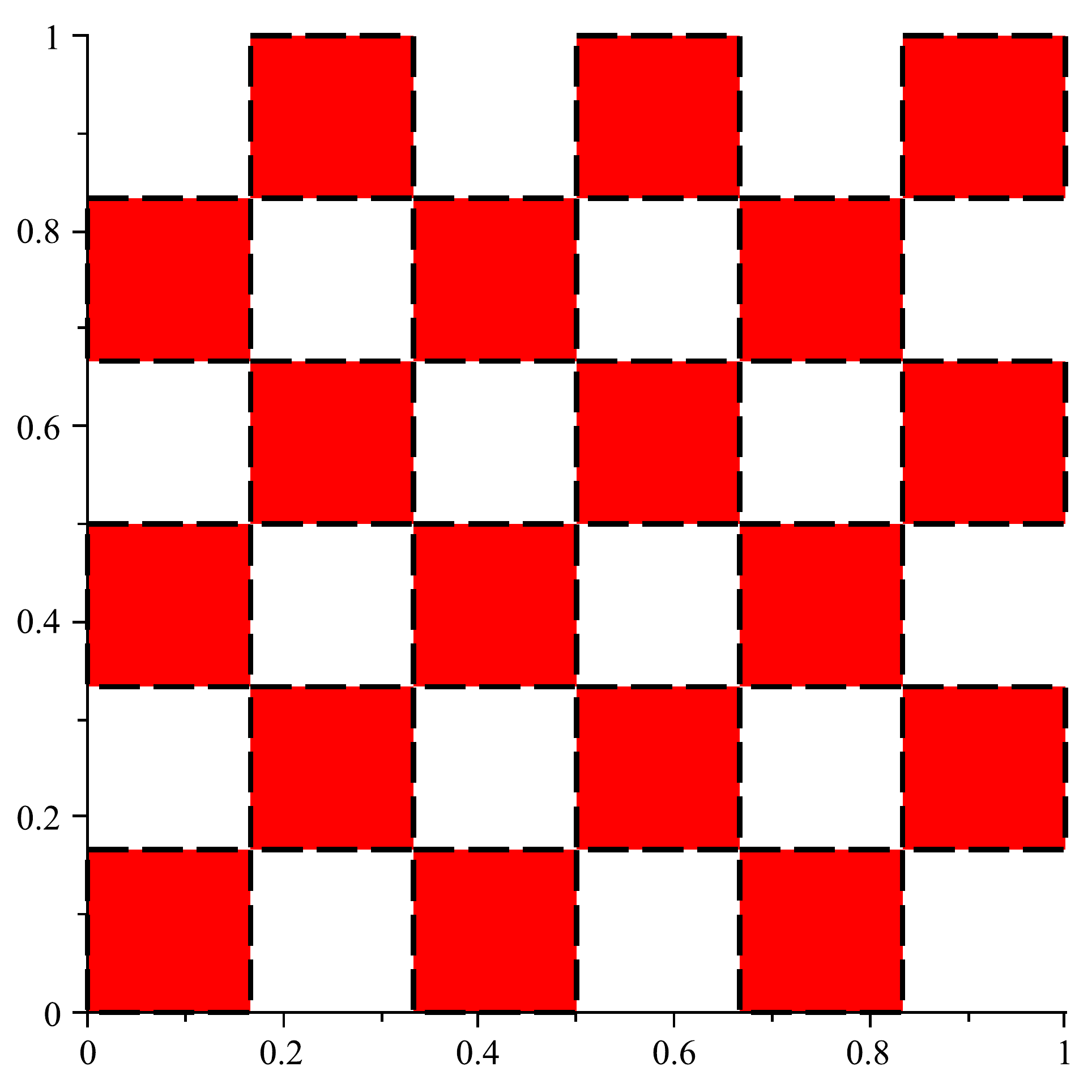}\hskip1cm\includegraphics[width=4cm]{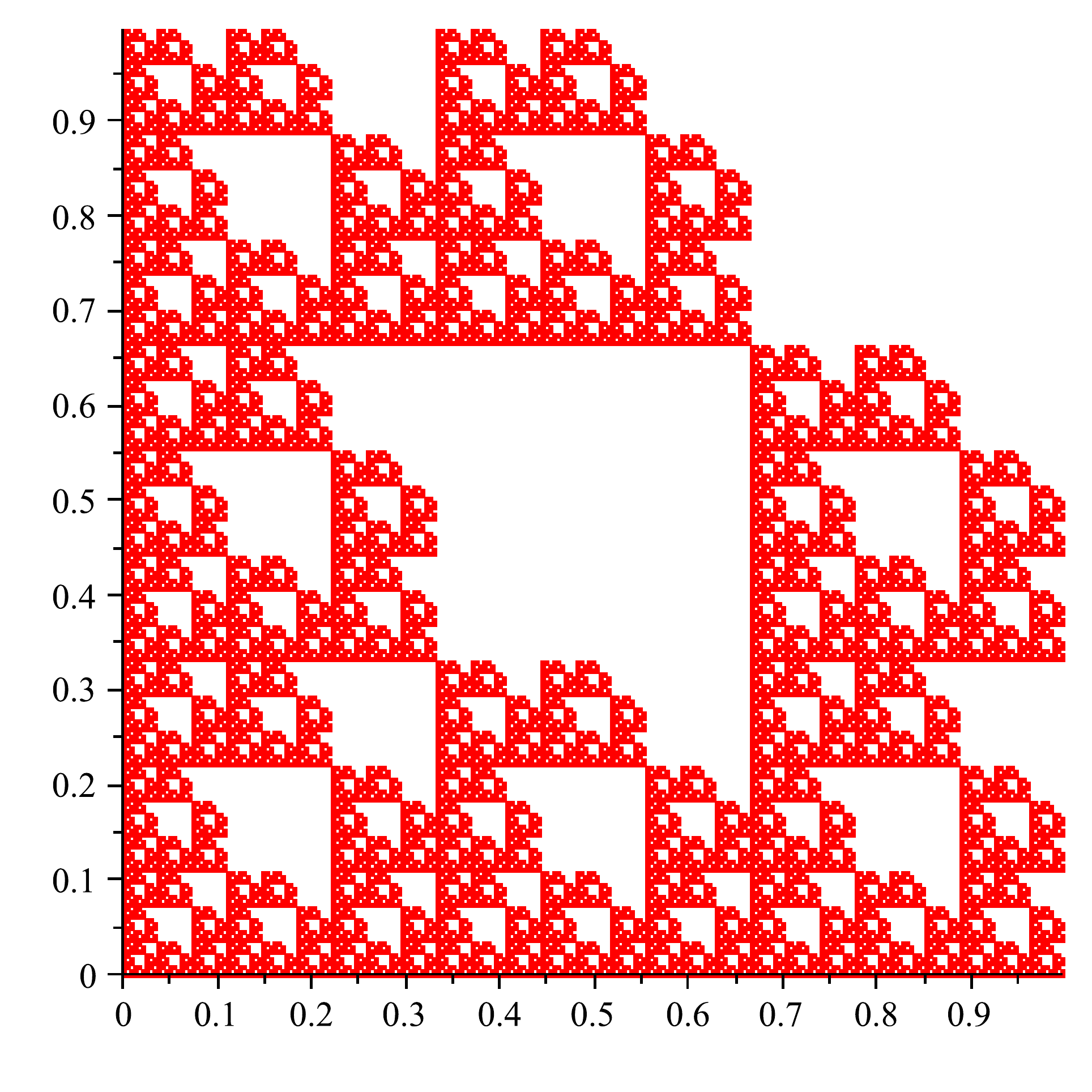}\\

On the left side: an example of an IFS attractor with irreducible cut sets
with one, two, three, four and five elements. On the right side: an example of an IFS attractor with an infinite irreducible cut set that is not perfect.
 \label{SquaresFiniteCut} \label{TrianglesNonPerfect}

\end{figure}

It is easily seen that the set
\[
T_n := \bigcup_{\begin{subarray}{c} i,j \in \{0,1,\ldots, n-1\} \\
i+j \equiv 0 \pmod{2}
\end{subarray}} Q_n(i,j)
\]
is the attractor of an IFS with $n^2$ similarities (the set $T_6$
is depicted on the left side of Figure~\ref{SquaresFiniteCut}). Moreover the set
\[
X :=
\left\{\left(\frac1n,\frac2n\right),\left(\frac2n,\frac3n\right),
\ldots,\left(\frac{n-2}n,\frac{n-1}n\right)\right\}
\]
forms an irreducible cut with cardinality $\# X = n-2$ and
\[
Y :=
\left\{\left(\frac{n-1}n,\frac1n\right),\left(\frac{n-2}n,\frac2n\right),
\ldots,\left(\frac{1}n,\frac{n-1}n\right)\right\}
\]
forms an irreducible cut with $\#Y = n-1$.

Since $n$ can be chosen to be any even positive integer we
conclude that for each given $m\in \mathbb{N}$ there exists an IFS
attractor having an irreducible cut $X$ with $\#X = m$.

We mention that $T_n$ even tiles the plane with respect to the
ornament group $p2$.
\end{example}

\begin{example}
Define the IFS
$\{f_{0,0},f_{1,0},f_{2,0},f_{0,1},f_{2,1},f_{0,2},f_{1,2}\}$ with
\[
f_{i,j}(x,y) := \left(\frac{x+i}3,\frac{y+j}3\right).
\]
The attractor $T$ of this IFS is depicted on the right side of
Figure~\ref{TrianglesNonPerfect}.


It is easy to see that the set
\[
X := \left\{\left( \frac{1}{3^n},\frac{1}{3^n}\right) : n \ge 1
\right\} \cup \left\{\left( \frac{2}{3^n},\frac{2}{3^n}\right) : n
\ge 1 \right\}
\]
is an irreducible cut of $T$. This set is infinite but not
perfect. On the other hand one easily verifies that $T$ has no
finite cuts.
\end{example}

We mention that the Sierpi\'nski triangle also has non-perfect irreducible cut sets.

\end{section}

\begin{section}{Cut sets of ${\mathbb Z}^d$-tiles}\label{SATIrrPerf}

We are interested in self-affine tiles, defined as attractors with nonempty interior of IFS of the form 

$$ \{f_e(x) = A^{-1}(x+e), x\in\mathbb{R}^d \}_{e\in\mathcal{D}},$$
where $A$ is a $d\times d$ matrix with eigenvalues greater than $1$ in modulus and $\mathcal{D}\subset \mathbb{R}^d$ with $ | \mathcal{D} | =  | \textrm{det} A| $ is supposed to be an integer.

In particular, we will be concerned with \emph{integral self-affine tiles with standard digit set}: this means that $A$ is an integer matrix and $\mathcal{D}\subset \mathbb{Z}^d$ is a complete set of coset representatives of $\mathbb{Z}^d/ A\mathbb{Z}^d$. Moreover the set $T+\mathbb{Z}^d$ will be assumed to be a tiling of $\mathbb{R}^d$, \emph{i.e.}, $ \mathbb{R}^d = T+\mathbb{Z}^d$
with  $(\textrm{int}(T+d_1))\cap (T+d_2)  = \emptyset$ for $d_1\neq d_2$  $(d_1,d_2\in\mathbb{Z}^d)$. We say also that $T = T(A,\mathcal{D})$ is a $   \mathbb{Z}^d$-tile in $\mathbb{R}^d$.

In the more restrictive setting of integral self-affine $\mathbb{Z}^d$-tiles,
Theorem~\ref{general} and its corollary take the following form.

\begin{theorem}\label{tile}
Let $T$ be a connected ${\mathbb Z}^d$-tile. Then one
of the following alternatives holds:
\begin{itemize}
\item $T$ has a cut point.
\item Each irreducible cut set $X$ of $T$ admits a partition
$X=X'\cup X_1$ with $X_1$ contained in
\[
\widetilde{E}:=\bigcup_{k\in\mathbb{N}}\bigcup_{ \begin{array}{c}v,s\in{\mathbb Z}^d,\\
\#T\cap (T+s)=1\end{array}}A^{-k}(T\cap (T+s))+A^{-k}v.
\]
\end{itemize}
\end{theorem}

The following corollary is an immediate consequence of this result.

\begin{corollary}
\label{cor1}
If $T$ is a connected ${\mathbb Z}^d$-tile and if $\#(T\cap
(T+s))\ne1$ for all $s\in\mathbb{Z}^d$,
 then either  $T$ has a cut point or each irreducible cut set of $T$ is a perfect set.
\end{corollary}

\begin{proof}[Proof for Theorem~\ref{tile}]: To prove this result just note that
$E\subset \widetilde{E}$ where $E$ is defined as in Theorem \ref{general}. The theorem
now is a consequence of Theorem~\ref{general}.
\end{proof}

\begin{remark}
Note that the conditions of these assertions are much easier to
check than the conditions of the more general result in
Theorem~\ref{general}. Moreover, since $T$ is compact, the quantities $\#(T\cap (T+s))$ can be non-zero for only finitely many $s\in\mathbb{Z}^d$ and can be read off from the
so-called {\em neighbor graph} associated to the ${\mathbb
Z}^d$-tile $T$, which can be computed algorithmically  from the data $(A,\mathcal{D})$ (see~\cite{ScheicherThuswaldner03}).
\end{remark}

\begin{example}\label{ExA4B5}
Let $T = T(A,\mathcal{D}) $ be the $\mathbb{Z}^2$-tile with

\[A=
\left(\begin{array}{cc}0&-5\\1&-4\end{array}\right),{\mathcal D}=\left\{\left(\begin{array}{c}i\\0\end{array}
\right); 0\le i\le 4\right\}
\]
This tile is the fundamental domain of the canonical number system associated with the complex basis $-2+i$ (see~\cite{KK}).  It  is depicted in Figure~\ref{A4B5Tiling} together with its  $10$ neighbors.

\begin{figure}[h]
\includegraphics[width=5cm]{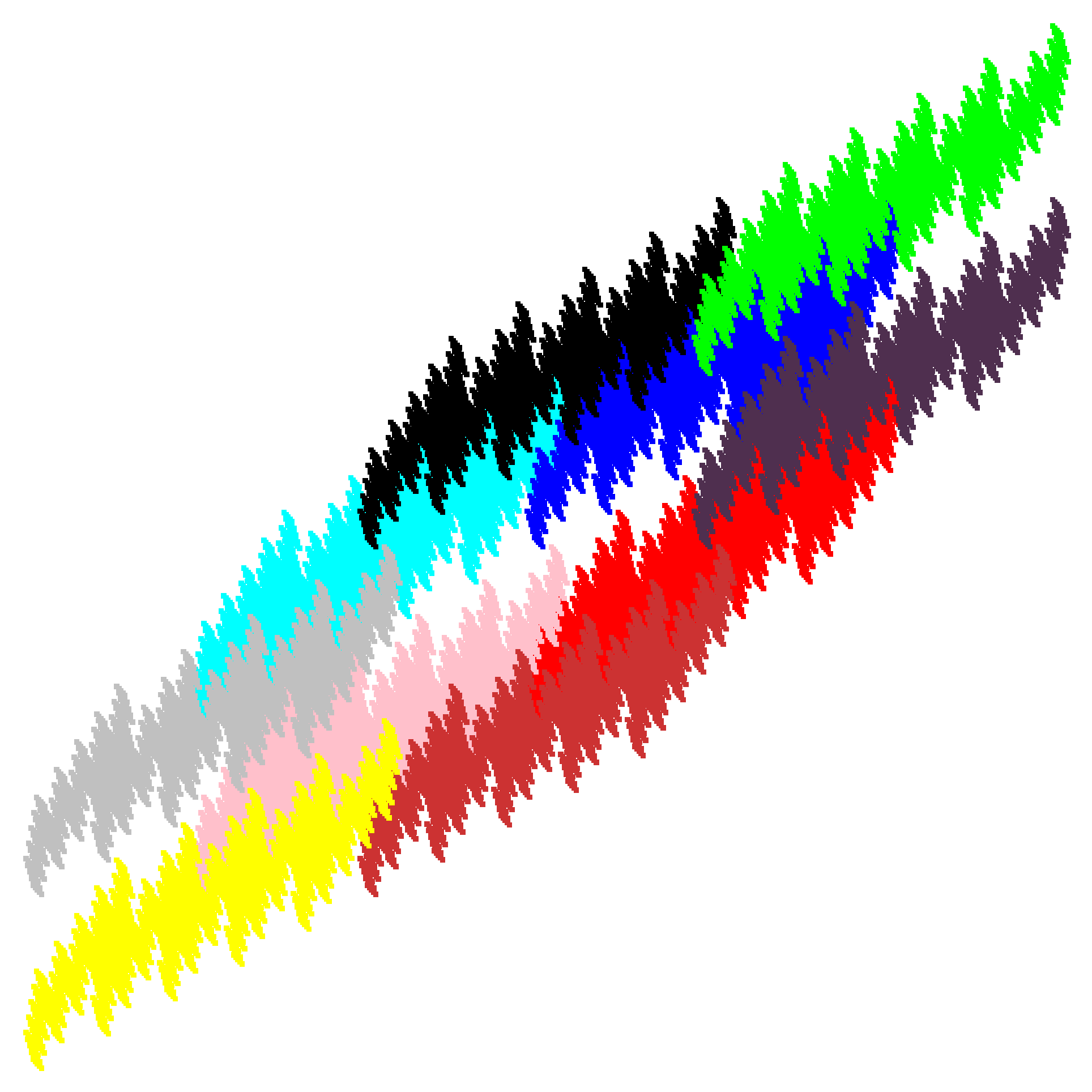}
\caption{Tiling A=4 B=5. \label{A4B5Tiling}}
\end{figure}

$T$ is a  connected self-affine tile~\cite{AT,AT1}.
In~\cite{NT1}, S-M. Ngai and T-M. Tang  used a technique developed in their paper to  
prove that $T$ has no cut point.  On the other hand, in~\cite{AT1}, S. Akiyama and J. M. Thuswaldner show  that $\#(T\cap (T+s))\ne 1$   for each $s\in\mathbb{Z}^d$.  Applying  Corollary \ref{cor1}, we deduce that each irreducible cut set of $T$ is a perfect set.

\end{example}
\end{section}

\begin{section}{A criterion for the existence of cut points of $\mathbb{Z}^d$-tiles}\label{sec:SATCuts}

In this section we will establish a theorem stating that under
certain circumstances the existence of a cut point of a
$\mathbb{Z}^d$-tile can be seen form the structure of the Hata
graphs $G_n$. The
formulation of the theorem will require some technical
preliminaries, however, as will be illustrated by two examples, its
application is often very simple.

For ${\mathbb Z}^d$-tiles $T$, the vertex set $V_n$ of the Hata
graph $G_n=(V_n,E_n)$ of $n$-th order associated to $T$ can be
defined in terms of elements of $\mathbb{Z}^d$ rather than strings
of digits. Indeed, for an integral self-affine $\mathbb{Z}^d$-tile $T(A,\mathcal{D})$ with standard digit set $\mathcal{D}$,  we get the following
equivalent representation of $G_n$.

\begin{itemize}
\item $V_n$: $v\in{\mathbb Z}^d$ is a vertex of $G_n$ if and only if
$v=\sum\limits_{k=0}^{n-1}A^kd_k$ with $d_k\in{\mathcal D}$ ($0\le
k\le n-1$).
\item $E_n$: there is an edge between $v$ and $v'$ if and only if
$(T+v)\cap( T+v')\ne\emptyset$.
\end{itemize}

\begin{remark}
$V_n$ is exactly the set of all $v\in{\mathbb Z}^d$ satisfying
$A^{-n}(T+v)\subset T$. Indeed $T$ can be written as
\[
T=\bigcup\limits_{v\in V_n}A^{-n}(T+v).
\]
\end{remark}

\begin{definition}[Children]
For $n\ge 1$ let $G_n=(V_n,E_n)$ be the Hata graphs associated to a
$\mathbb{Z}^d$-tile $T$. A node $v'\in V_{n+1}$ is called a
\emph{child} of a node $v\in V_n$ if
\[
A^{-n-1}(T+v') \subset A^{-n}(T+v).
\]
If $U \subseteq V_n$ is given, then
\[
C(U) :=AU+\mathcal{D}= \{v' \in V_{n+1} \;;\; \exists v\in U: A^{-n-1}(T+v')
\subset A^{-n}(T+v)\} 
\]
is the \emph{set of children} of the elements of $U$. Note that
$C(U) \subseteq V_{n+1}$.
\end{definition}

In what follows, we will consider cut sets of the Hata graphs $G_n$.
To this matter we will study partitions $V_n=W_1^{(n)} \cup
W_2^{(n)} \cup W_3^{(n)}$ of the Hata graph with the property that
$W_2^{(n)}$ separates $G_n$ between $W_1^{(n)}$ and $W_3^{(n)}$, {\it
i.e.}, all paths in $G_n$ between a node $v_1 \in W_1^{(n)}$ and a node  $v_3
\in W_3^{(n)}$ pass through nodes contained in $W_2$. Set
\[
X_n(W) := \bigcup_{v \in W} A^{-n}(T+v) \qquad (W \subseteq V_n).
\]
Then, by the definition of the Hata graph the separation property of
$W_2^{(n)}$ in $G_n$ means that $X_n(W_2^{(n)})$ separates between
$X_n(W_1^{(n)})$ and $X_n(W_3^{(n)})$, {\it i.e.}, since $T =
X_n(W_1^{(n)})\cup X_n(W_2^{(n)})\cup X_n(W_3^{(n)})$ this means that
$X_n(W_2^{(n)})$ forms a separation of $T$. We will define
$W_i^{(n)}$ in a way that there exists a single point $z\in T$ such
that
\begin{equation}\label{inclusion}
X_{n}(W_2^{(n)}) \supset X_{n+1}(W_2^{(n+1)})\quad \hbox{and}\quad
\bigcap_{n\ge n_0} X_{n}(W_2^{(n)}) = \{z\}.
\end{equation}
The point $z$ will turn out to be a cut point. Note that the
inclusion relation on the left hand side of \eqref{inclusion} is
equivalent to $W_2^{(n+1)} \subset C(W_2^{(n)})$.

This procedure will \emph{a priori} require to check infinitely many
partitions $W_1^{(n)}$, $W_2^{(n)}$, $W_3^{(n)}$. However, we will see in the examples that the self-affinity of the tiles allows to define these partitions in a recursive way. We are now in a position to state the following theorem.

\begin{theorem}[Cut points of $\mathbb{Z}^d$-tiles]\label{cutpoint}
Let $T$ be a connected $\mathbb{Z}^d$-tile and $G_n=(V_n,E_n)$
($n\in\mathbb{N}$) its associated Hata graphs. Let $n_0$ be a 
positive integer. Suppose that for each $n \geq n_0$ there are partitions $V_n=W_1^{(n)} \cup
W_2^{(n)} \cup W_3^{(n)}$ with $W_j^{(n)}\not=\emptyset$ $(j\in\{1,2,3\})$
satisfying the following properties.
\begin{itemize}
\item[$(i)$] The set $W_2^{(n)}$ induces a connected subgraph of $G_n$ that separates $G_n$ between $W_1^{(n)}$ and
$W_3^{(n)}$.
\item[$(ii)$] The inclusions
$$\begin{array}{l}
\bullet \;W_1^{(n+1)} \subset C(W_1^{(n)}) \cup C(W_2^{(n)}), \\\\
\bullet \;W_2^{(n+1)}\subset C(W_2 ^{(n)})\\\\
\bullet\; W_3^{(n+1)} \subset C(W_3^{(n)}) \cup C(W_2^{(n)})
\end{array}$$
hold. 
\item[$(iii)$] The number of elements of $W_2^{(n)}$
 is uniformly bounded, \emph{i.e.}, there is an integer $K$ such that for all $n\in\mathbb{N}$, we have
 $\#W_2^{(n)}\leq K.
 $ 
\end{itemize}
Then $\bigcap_{n\geq n_0}X_n(W_2^{(n)})$ consists of a single point and this point is a cut point of $T$. 
\end{theorem}

\begin{remark} The connectedness of a self-affine $\mathbb{Z}^d$-tile can be checked algorithmically~\cite{H}.
\end{remark}

\begin{proof} 
 
We first show that $\bigcap_{n\geq n_0}X_n(W_2^{(n)})$ is a single point. By the second assumption in Item $(ii)$,  we have for all $n\geq n_0$ that
$$X_n(W_2^{(n)}) \supset X_{n+1}(W_2^{(n+1)})\ne\emptyset.
$$
Since $T$ is connected and $W_2^{(n)}$ induces a connected subgraph in the Hata graph, we conclude that $X_n(W_2^{(n)})$ is a compact connected subset of $T$ for all $n\geq n_0$. Now, by Item~$(iii)$, for all $n\geq n_0$,
$$\textrm{diam}(X_n(W_2^{(n)}))\leq \frac{K\;\textrm{diam}(T)}{\lambda^n},
$$
where $\lambda$ is the smallest eigenvalue of $A$ in absolute value. Indeed, $X_n(W_2^{(n)})$ is a union of subtiles of the form $A^{-n}(T+v)$ for some $v\in V_n$, whose diameters are smaller than $\textrm{diam}(T)/\lambda^n$. Moreover, by the connectedness of the subgraph induced by $W_2^{(n)}$ in the Hata graph, any two points $x,y$ of $X_n(W_2^{(n)})$ can be connected by a connected chain of at most $K$ such subtiles, \emph{i.e.}, there are $v_1,\ldots,v_k\in V_n$ with $k\leq K$ such that 
$x\in A^{-n}(T+v_1) ,\;y\in A^{-n}(T+v_k)$ and 
$$A^{-n}(T+v_j) \cap A^{-n}(T+v_{j+1}) \ne\emptyset \textrm{ for }j\in\{1,\ldots, k-1\}.
$$ 
This implies the above inequality. Therefore, $(X_n(W_2^{(n)}))$ is a closed nested sequence of compact sets whose diameters converge to $0$ as $n\to\infty$, and the intersection $\bigcap_{n\geq n_0} X_n(W_2^{(n)})$ consists of a single point. 

We call this single point $z$. We shall prove now that $z$ is a cut point of $T$. Let 
$$X:=\bigcup_{n\geq n_0} X_n(W_1^{(n)})\;\setminus\;\{z\},\; Y:= \bigcup_{n\geq n_0} X_n(W_3^{(n)})\;\setminus\;\{z\}.
$$ Obviously, $X\cap Y=\emptyset$. Moreover, since $V_n=W_1^{(n)} \cup
W_2^{(n)} \cup W_3^{(n)}$  is a partition of the Hata graph, we have the  following inclusions by Item $(ii)$ for all $n\geq n_0$:
$$\begin{array}{l}
\bullet \;\emptyset\ne X_n(W_1^{(n)} )\subset X_{n+1}(W_1^{(n+1)}),\\\\
\bullet\;\emptyset\ne X_n(W_3^{(n)} )\subset X_{n+1}(W_3^{(n+1)}).
\end{array}$$
In particular, $X\ne \emptyset\ne Y$. Also, $z\notin X\cup Y$. Thus we infer the existence of the partition 
$$T=X\cup\{z\}\cup Y.
$$ 
More precisely, let us show that $T$ is connected while $T\setminus\{z\}=X\cup Y$ is a separation, that is, $\overline{X}\cap Y=\emptyset=X\cap \overline{Y}$. Indeed, suppose on the contrary that $y\in \overline{X}\cap Y$. Then there is $N\geq n_0$ such that $y\in X_N(W_3^{(N)})$. Since $W_2^{(n)}$ separates the Hata graph $G_n$ between $W_1^{(n)}$ and $W_3^{(n)}$ for all $n\geq n_0$, we conclude that for all $n\geq N$, $y\notin X_n(W_1^{(n)})$. Also, as $y\ne z$, there exists $N_1\geq N$ such that  $y\notin X_{N_1}(W_2^{(N_1)})$. Now, 
$$\begin{array}{rcl}
\overline{X}\subset\overline{\bigcup_{n\geq n_0}X_n(W_1^{(n)})}&=&\overline{\bigcup_{n\geq N_1}X_n(W_1^{(n)})}\\
&\subset& \overline{X_{N_1}(W_1^{(N_1)})\cup X_{N_1}(W_2^{(N_1)})}\\
&=&X_{N_1}(W_1^{(N_1)})\cup X_{N_1}(W_2^{(N_1)}).
\end{array}
$$
The second inclusion above follows from the fact that, by Item $(ii)$, for all $n\geq N_1$, 
$$X_n(W_1^{(n)})\subset X_{N_1}(W_1^{(N_1)})\cup X_{N_1}(W_2^{(N_1)}).
$$ 
 However, by construction of $N_1$, $y$ is neither in $X_{N_1}(W_1^{(N_1)})$, nor in $X_{N_1}(W_2^{(N_1)})$. Thus $y\notin \overline{X}$, a contradiction. Hence we obtained that $\overline{X}\cap Y=\emptyset$.

Since we can obtain similarly that $X\cap \overline{Y}=\emptyset$, we conclude that $T\setminus\{z\}=X\cup Y$ is a separation. In other words,  $z$ is a cut point of $T$.

\end{proof}

We apply our theorem to two examples.

\begin{example}\label{ExBandtGelbrich}
We consider as a first example the tile $T$ satisfying the equation
$$AT=T+\mathcal{D}=\bigcup_{e\in\mathcal{D}}(T+e),
$$
where 
$$A=\left(\begin{array}{cc}0&3\\1&1\end{array}\right),\;\mathcal{D}=\left\{\left(\begin{array}{c}0\\0\end{array}\right),\left(\begin{array}{c}1\\0\end{array}\right),\left(\begin{array}{c}-1\\0\end{array}\right)\right\}.
$$
This tile can be found in~\cite{BandtGelbrich94} and~\cite[Figure 4]{BW}. We depict it with its neighbors on Figure~\ref{ExBG}.
\begin{figure}[h]
\includegraphics[width=7cm,height=7cm]{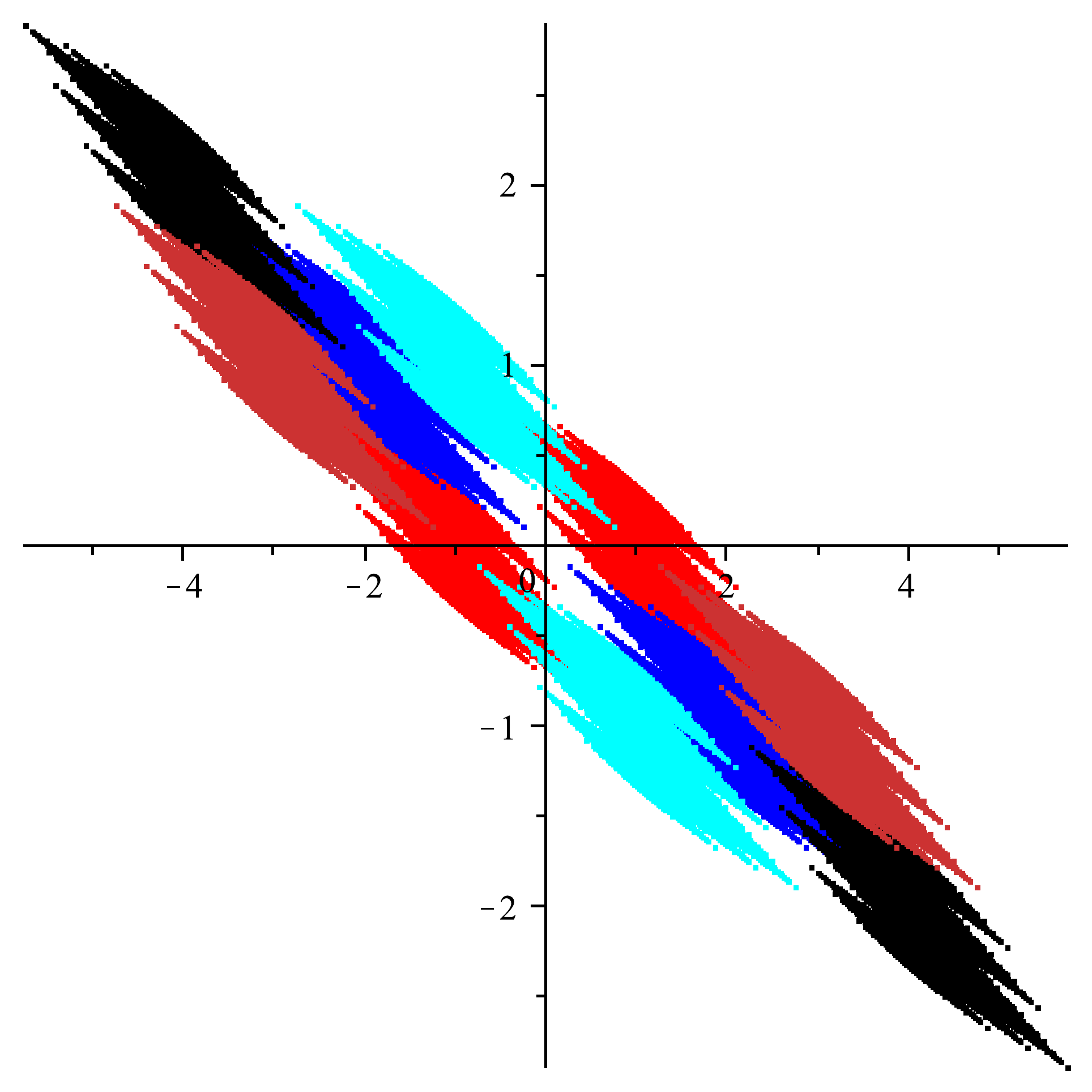}
\caption{Tiling for Example~\ref{ExBandtGelbrich}. \label{ExBG}}
\end{figure}
The tile $T$ has the following neighbors: 
\begin{equation}\label{SBandtGelbrichEx}
\mathcal{S} = \left\{\pm  \left(\begin{array}{c}1\\0\end{array}\right),\pm\left(\begin{array}{c}-2\\1\end{array}\right),
\pm \left(\begin{array}{c}-1\\1\end{array}\right),\pm\left(\begin{array}{c}-4\\2\end{array}\right),\pm\left(\begin{array}{c}-3\\1\end{array}\right)\right\}.
\end{equation}
This can be computed via well-known algorithms (see for example~\cite{ScheicherThuswaldner03}).

For $n\in\mathbb{N}$, consider the associated $n$-th Hata graph $G_n = (V_n,E_n)$. We write the elements of $V_n$ as strings of digits, \emph{i.e.}, $d_{n-1}\cdots d_0\in\mathcal{D}^\mathbb{N}$ stands here for $\sum_{k=0}^{n-1}A^kd_k$. For convenience, we will often write $d$ instead of $\left(\begin{array}{c}d\\0\end{array}\right)$ ($d\in\{-1,0,1\}$). If a digit $d\in\mathcal{D}$ is repeated  $p$ times for some $p\geq 0$, we will simply write $d^p$ instead of $\underbrace{d\cdots d}_{p\textrm{ times}}$. For $p=0$, this is just the empty word. 
 
 We check the assumptions of Theorem~\ref{cutpoint} to obtain the existence of a cut point. We define the sets $W_i^{(n)}$ for $i\in\{1,2,3\}$ and $n\geq n_0:=1$ recursively as follows. Let 
$$W_1^{(1)}  = \{-1\},\;W_2^{(1)} = \{0\},\; W_3^{(1)} = \{1\}
$$
and for $n\geq 1$, let 
 $$\begin{array}{rcl}
  W_1^{(n+1)} &=& (AW_1^{(n)}+\mathcal{D})
  \cup\{0^{n}(-1)\}\\
  
 W_2^{(n+1)} &=& \{0^{n+1}\} ,\\  
W_3^{(n+1)} &=& (AW_3^{(n)}+\mathcal{D})
  \cup\{0^n1\}.
\end{array}
$$
\begin{remark} For $n\in\mathbb{N}$, we have $0^n=0=\left(\begin{array}{c}0\\0\end{array}\right)$, but we want to stress the fact that we are considering elements of $V_n$. 
\end{remark}
The sets $W_i^{(n)}$ ($i\in\{1,2,3\}$) partition $V_n$ for all $n\geq 1$. Moreover, they satisfy  Conditions $(ii)$ and $(iii)$ of Theorem~\ref{cutpoint} with $K=1$. As $W_2^{(n)}$ consists of a single point, it is connected for all $n\geq 1$. Let us now prove the separation property of Condition $(i)$ by induction. More precisely, we shall prove the following lemma. 
\begin{lemma}
 For all $n\geq 1$,
\begin{itemize}
\item[$(a)$] $W_2^{(n)}$ separates $G_n$ between $W_1^{(n)} $ and $W_3^{(n)}$;
\item[$(b)$] 
$$\{v\in W_1^{(n)};(T+v)\cap (T+0^n)\ne\emptyset\}=\{0^{n-1}(-1),\underbrace{0^{n-2}(-1)1}_{\textrm{only for }n\geq 2}\}$$
and 
$$\{v\in W_3^{(n)};(T+v)\cap (T+0^n)\ne\emptyset\}=\{0^{n-1}1,\underbrace{0^{n-2}1(-1)}_{\textrm{only for }n\geq 2}\}.
$$
\end{itemize}

\end{lemma}

 \begin{proof}
 The first Hata graph is shown in Figure~\ref{H1GraphBandtGelbrichEx}, the second Hata graph in Figure~\ref{H2GraphBandtGelbrichEx}. Note that,  by definition, $G_1$ has the set of vertices $V_1=\mathcal{D}=\{-1,0,1\}$, and 
 $V_2=A\mathcal{D}+\mathcal{D}$, \emph{i.e.},  $G_2$ has the set of vertices
  $$\begin{array}{rcl}
  V_2&=& \{v\in\mathbb{Z}^2;  v = \sum_{k=0}^1 A^k d_k, d_k\in\mathcal{D}\}\\
  & =& \{(-1)(-1),(-1)0,(-1)1,0(-1),00,01,1(-1),10,11 \}.
  \end{array}
  $$
  The edges are constructed using the neighbor relations of~\eqref{SBandtGelbrichEx}.

 We read from these pictures that Items $(a)$ and $(b)$ are thus true for $n\in\{1,2\}$.  Suppose now that Items  $(a)$ and $(b)$ hold true for some $n\geq 2$.  By $(a)$, using the set equation for $T$, we have
 $$(T+W_1^{(n)})\cap(T+W_3^{(n)})=\emptyset\Rightarrow (T+AW_1^{(n)}+\mathcal{D})\cap(T+AW_3^{(n)}+\mathcal{D})=\emptyset.
 $$ 
Also, 
 $$(T+0^n(-1))\cap(T+0^n1)\ne\emptyset\iff 0^n(-1)-0^n1=\left(\begin{array}{c}-1-1\\0\end{array}\right)\in\mathcal{S}.
 $$
  However, this does not happen, as can be seen from~\eqref{SBandtGelbrichEx}. Furthermore, 
   if  $w=w_1d_0\in A W_1^{(n)}+\mathcal{D}$ and $w'=0^n1$ with $w_1\in W_1^{(n)}$ and $d_0\in\mathcal{D}$ satisfy 
 $$(T+w)\cap (T+w')\ne\emptyset,
 $$
 then by the induction hypothesis  on Item $(b)$ we  have $w_1\in\{0^{n-1}(-1),0^{n-2}(-1)1\}$, and 
 $$(T+0^{n-1}(-1)d_0)\cap(T+0^{n-1}01)\ne\emptyset\iff (-1)d_0-01= \left(\begin{array}{c}d_0-1\\-1\end{array}\right)\in\mathcal{S};
 $$
 This only happens for $d_0-1\in\{1,2,3\}$, as can be seen from~\eqref{SBandtGelbrichEx}, but no choice of $d_0\in\{-1,0,1\}$ leads to these values. 
 Similarly, 
  $$(T+0^{n-2}(-1)1d_0)\cap(T+0^{n-2}001)\ne\emptyset\iff (-1)1d_0-001= \left(\begin{array}{c} -4+d_0\\ 0 \end{array}\right)\in\mathcal{S}.
 $$
 This does not happen for $d_0\in\{-1,0,1\}$, as can be seen from~\eqref{SBandtGelbrichEx}. In the same way, we can show that $(T+w)\cap (T+0^n(-1))=\emptyset$ for all $w\in AW_3^{(n)}+\mathcal{D}$. This proves that $W_2^{(n+1)}$ separates $G_{n+1}$ between $W_1^{(n+1)} $ and $W_3^{(n+1)}$ (Item $(a)$). For Item $(b)$, just note that 
$$\forall d\in\mathcal{D}, \;(T+0^n0)\cap (T+0^nd)\ne\emptyset\iff 0-d=\left(\begin{array}{c}0-d\\0\end{array}\right)\in\mathcal{S}\iff d\in\{-1,1\}. 
$$
 And if $w=w_1d_0\in A W_1^{(n)}+\mathcal{D}$ with $w_1\in W_1^{(n)}$ and $d_0\in\mathcal{D}$, then as above 
$$\begin{array}{rcl}
(T+w)\cap (T+0^{n}0)\ne\emptyset&\Rightarrow& w_1\in\{0^{n-1}(-1),0^{n-2}(-1)1\}\\
& \Rightarrow& 0^{n-1}(-1)d_0-0^{n-1}00= \left(\begin{array}{c}d_0\\-1\end{array}\right)\in\mathcal{S} \\
&&\textrm{ or } 0^{n-2}(-1)1d_0-0^{n-2}000= \left(\begin{array}{c}-3+d_0\\0\end{array}\right)\in\mathcal{S} 
\\
&\Rightarrow & w=0^{n-1}(-1)1
\end{array}
 $$
(compare with~\eqref{SBandtGelbrichEx}). Similarly, $(T+w)\cap (T+0^{n+1})\ne\emptyset$ for $w\in AW_3^{(n)}+\mathcal{D}$ occurs only possible for $w=0^{n-1}1(-1)$.

 \end{proof}

Therefore, we can apply Theorem~\ref{cutpoint} and obtain that
$$\bigcap_{n\geq 1}X_n(W_2^{(n)}) =\{0.0^\infty\},
$$
where $0.0^\infty=\sum_{k=1}^\infty A^{-k}\left(\begin{array}{c}0\\0\end{array}\right)=\left(\begin{array}{c}0\\0\end{array}\right)$ is a cut point of $T$.

\begin{figure}[h]
\begin{center}

\includegraphics[width=4cm,height=1cm]{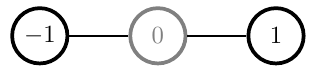}
\vskip 0.1cm
\includegraphics[width=6cm,height=5cm]{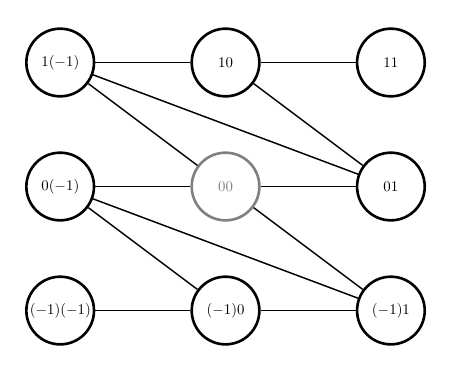}
\end{center}
\caption{First and second Hata Graphs for Example~\ref{ExBandtGelbrich}.}
\label{H1GraphBandtGelbrichEx}
\label{H2GraphBandtGelbrichEx}
\end{figure}

\end{example}

\begin{example}\label{ExCNSA6B7}
As a second example, we consider a planar tile associated with a canonical number system (\emph{CNS-tile}). It is defined similarly as in Example~\ref{ExA4B5} by $T=T(A,\mathcal{D})$ with

\[A=
\left(\begin{array}{cc}0&-7\\1&-6\end{array}\right),{\mathcal D}=\left\{\left(\begin{array}{c}i\\0\end{array}
\right); 0\le i\le 6\right\}.
\]

It is shown in~\cite{AT1} that this tile is not homeomorphic to a disk. We obtain here further information and exhibit a cut point of $T$.

The set of neighbors of $\T$ can be found in~\cite{AT1}:
\begin{equation}\label{SA6B7}
\mathcal{S}=\left\{\pm P_1,\pm P_2,\pm P_3,\pm Q_1,\pm Q_2,\pm Q_3,\pm R\right\},
\end{equation}
where
$$P_m=\left(\begin{array}{c}-5m+6\\-m+1\end{array}\right),Q_m=\left(\begin{array}{c}5m\\m\end{array}\right),R=\left(\begin{array}{c}-6\\-1\end{array}\right) \;(m\in\{1,2,3\}).
$$

We check the assumptions of Theorem~\ref{cutpoint} to obtain the existence of a cut point. To this effect, we define the sets $W_i^{(n)}$ for $i\in\{1,2,3\}$ and $n\geq n_0:=1$ recursively as follows. Let 
$$W_1^{(1)}  = \{0,1,2\},\;W_2^{(1)} = \{3\},\; W_3^{(1)} = \{4,5,6\}
$$
and for $n\geq 1$, let 
 $$\begin{array}{rcl}
  W_1^{(n+1)} &=& (AW_1^{(n)}+\mathcal{D})
  \cup\left\{\begin{array}{cl}\{3^{n}4,3^{n}5,3^{n}6\}&\textrm{ if } n \textrm{ is odd}\\\\
  \{3^{n}0,3^{n}1,3^{n}2\}&\textrm{ if } n \textrm{ is even}
  \end{array}
  \right., \\
W_2^{(n+1)} &=& \{3^{n+1}\} ,\\  
W_3^{(n+1)} &=& (AW_3^{(n)}+\mathcal{D})
  \cup\left\{\begin{array}{cl}\{3^{n}0,3^{n}1,3^{n}2\}&\textrm{ if } n \textrm{ is odd}\\\\
  \{3^{n}4,3^{n}5,3^{n}6\}&\textrm{ if } n \textrm{ is even}
  \end{array}
  \right.. 
\end{array}
$$
Note that the sets $W_i^{(n)}$ ($i\in\{1,2,3\}$) constitute a partition of $V_n$ for all $n\geq 1$. Moreover, they satisfy  Conditions $(ii)$ and $(iii)$ of the theorem with $K=1$. As $W_2^{(n)}$ consists of a single point, it is connected for all $n\geq 1$. Let us now prove the separation property of Condition $(i)$ by induction. More precisely, we shall prove the following lemma. 
\begin{lemma}
 For all $n\geq 1$,
\begin{itemize}
\item[$(a)$] $W_2^{(n)}$ separates $G_n$ between $W_1^{(n)} $ and $W_3^{(n)}$;
\item[$(b)$] 
$$\{v\in W_1^{(n)};(T+v)\cap (T+3^n)\ne\emptyset\}=\left\{\begin{array}{cl}\{3^{n-1}2\}&\textrm{ if } n \textrm{ is odd}\\\\
  \{3^{n-1}4\}&\textrm{ if } n \textrm{ is even}
  \end{array}
  \right.
$$
and 
$$\{v\in W_3^{(n)};(T+v)\cap (T+3^n)\ne\emptyset\}=\left\{\begin{array}{cl}\{3^{n-1}4\}&\textrm{ if } n \textrm{ is odd}\\\\
  \{3^{n-1}2\}&\textrm{ if } n \textrm{ is even}
  \end{array}
  \right..
$$
\end{itemize}

\end{lemma}

\begin{proof}
The first Hata graph can be seen in Figure~\ref{H1GraphA6B7}. We have $V_1=\mathcal{D}$ and the edges are constructed according to~\eqref{SA6B7}. As can be seen on this graph, $W_2^{(1)}$ separates $G_1$ between $W_1^{(1)}$ and $W_3^{(1)}$. Also, 
$$\{v\in W_1^{(1)};(T+v)\cap (T+3)\ne\emptyset\}=\{2\}\textrm{ and } \{v\in W_3^{(1)};(T+v)\cap (T+3)\ne\emptyset\}=\{4\}.
$$
Hence, Items $(a)$ and $(b)$ are fulfilled for $n=1$.

Now,  by definition, $V_2=A\mathcal{D}+\mathcal{D}$, \emph{i.e.},  $G_2$ has the set of vertices 
$$\begin{array}{rcl}
V_2&=& \{v\in\mathbb{Z}^2;  v = \sum_{k=0}^1 A^k d_k, d_k\in\mathcal{D}\} \\
&=& \{00,01,\ldots, 06, 10,11,\ldots, 16,20,21,\ldots,26, \ldots, 60,61,\ldots 66\},
\end{array}
$$
and 
$$\begin{array}{rcl}
  W_1^{(2)} &=& (AW_1^{(1)}+\mathcal{D})\cup\{34,35,36\}, \\
W_2^{(2)} &=& \{33\} ,\\  
W_3^{(2)} &=& (AW_3^{(1)}+\mathcal{D})\cup\{30,31,32\}.
\end{array}
$$
The following considerations can be directly seen from the second Hata graph $G_2$ depicted  in Figure~\ref{H2GraphA6B7}.
 First, using the set equation for $T=T(A,\mathcal{D})$, we have
 $$(T+W_1^{(1)})\cap (T+W_3^{(1)})=\emptyset\Rightarrow (T+AW_1^{(1)}+\mathcal{D})\cap(T+AW_3^{(1)}+\mathcal{D})=\emptyset.
 $$ 
Second,  
$$(T+3d)\cap(T+3d')\ne\emptyset\iff 3d-3d'=\left(\begin{array}{c}d-d'\\0\end{array}\right)\in\mathcal{S}.
 $$
  However, this does not happen for $d\in\{4,5,6\}$ and $d'\in\{0,1,2\}$, as can be seen from~\eqref{SA6B7}. Third, if  $w=d_1d_0\in A W_1^{(1)}+\mathcal{D}$ and $w'=3d$ with $d_1,d\in\{0,1,2\}$ and $d_0\in\mathcal{D}$ satisfy 
 $$(T+w)\cap (T+w')\ne\emptyset,
 $$
 then by Item $(b)$ for $n=1$ we conclude that $d_1=2$, and 
 $$(T+2d_0)\cap(T+3d)\ne\emptyset\iff 2d_0-3d= \left(\begin{array}{c}d_0-d\\2-3\end{array}\right)=\left(\begin{array}{c}d_0-d\\-1\end{array}\right)\in\mathcal{S}.
 $$
This only happens for $d_0-d\in\{-4,-5,-6\}$, as can be seen from~\eqref{SA6B7}, but no choice of $d_0$ and $d$ leads to these values. Similarly, $(T+w)\cap (T+w')=\emptyset$ for all $w\in AW_3^{(1)}+\mathcal{D}$ and $w'\in\{34,35,36\}$. This proves that Item $(a)$ is fulfilled for $n=2$. For Item $(b)$, just note that 
$$\forall d\in\mathcal{D}, \;(T+33)\cap (T+3d)\ne\emptyset\iff 33-3d=\left(\begin{array}{c}3-d\\0\end{array}\right)\in\mathcal{S}\iff d\in\{2,4\}. 
$$
 And if $w=d_1d_0\in A W_1^{(1)}+\mathcal{D}$ with $d_1\in\{0,1,2\}$ and $d_0\in\mathcal{D}$, then as above 
$$(T+w)\cap (T+33)\ne\emptyset\Rightarrow d_1=2 \textrm{ and } 2d_0-33= \left(\begin{array}{c}d_0-3\\-1\end{array}\right)\in\mathcal{S},
 $$
 which does not happen. Similarly, $(T+w)\cap (T+33)=\emptyset$ for $w\in AW_3^{(1)}+\mathcal{D}$.

 By the above considerations, for $n\in\{1,2\}$,  Items $(a)$ and $(b)$ are fulfilled. Suppose that  Items $(a)$ and $(b)$ hold true for some $n\geq 1$. Assume first that $n$ is odd, that is, 
 $$\begin{array}{rcl}
  W_1^{(n+1)} &=& (AW_1^{(n)}+\mathcal{D}) \cup\{3^{n}4,3^{n}5,3^{n}6\},\\
 W_2^{(n+1)} &=& \{3^{n+1}\} ,\\  
W_3^{(n+1)} &=& (AW_3^{(n)}+\mathcal{D})  \cup\{3^{n}0,3^{n}1,3^{n}2\}.
\end{array}
$$
 By the separation property, using the set equation for $T$, we have
 $$(T+W_1^{(n)})\cap(T+W_3^{(n)})=\emptyset\Rightarrow (T+AW_1^{(n)}+\mathcal{D})\cap(T+AW_3^{(n)}+\mathcal{D})=\emptyset.
 $$ 
Also, 
 $$(T+3^nd)\cap(T+3^nd')\ne\emptyset\iff 3^nd-3^nd'=\left(\begin{array}{c}d-d'\\0\end{array}\right)\in\mathcal{S}.
 $$
  However, this does not happen for $d\in\{4,5,6\}$ and $d'\in\{0,1,2\}$, as can be seen from~\eqref{SA6B7}. Furthermore, 
   if  $w=w_1d_0\in A W_1^{(n)}+\mathcal{D}$ and $w'=3^nd$ with $w_1\in W_1^{(n)}$, $d\in\{0,1,2\}$ and $d_0\in\mathcal{D}$ satisfy 
 $$(T+w)\cap (T+w')\ne\emptyset,
 $$
 then by the induction hypothesis  on Item $(b)$ we conclude that $w_1=3^{n-1}2$, and 
 $$(T+3^{n-1}2d_0)\cap(T+3^{n-1}3d)\ne\emptyset\iff 2d_0-3d= \left(\begin{array}{c}d_0-d\\2-3\end{array}\right)=\left(\begin{array}{c}d_0-d\\-1\end{array}\right)\in\mathcal{S}.
 $$
This only happens for $d_0-d\in\{-4,-5,-6\}$, as can be seen from~\eqref{SA6B7}, but no choice of $d_0$ and $d$ leads to these values. Similarly, $(T+w)\cap (T+w')=\emptyset$ for all $w\in AW_3^{(n)}+\mathcal{D}$ and $w'\in\{34,35,36\}$. This proves that $W_2^{(n+1)}$ separates $G_{n+1}$ between $W_1^{(n+1)} $ and $W_3^{(n+1)}$ (Item $(a)$). For Item $(b)$, just note that 
$$\forall d\in\mathcal{D}, \;(T+3^n3)\cap (T+3^nd)\ne\emptyset\iff 3-d=\left(\begin{array}{c}3-d\\0\end{array}\right)\in\mathcal{S}\iff d\in\{2,4\}. 
$$
 And if $w=w_1d_0\in A W_1^{(n)}+\mathcal{D}$ with $w_1\in W_1^{(n)}$ and $d_0\in\mathcal{D}$, then as above 
$$(T+w)\cap (T+3^{n-1}33)\ne\emptyset\Rightarrow w_1=3^n2 \Rightarrow 2d_0-33= \left(\begin{array}{c}d_0-3\\-1\end{array}\right)\in\mathcal{S},
 $$
 which does not happen. Similarly, $(T+w)\cap (T+3^{n+1})=\emptyset$ for $w\in AW_3^{(n)}+\mathcal{D}$.

The argument runs similarly as above if $n$ is even.

\end{proof}

Therefore, we can apply Theorem~\ref{cutpoint} and obtain that
$$\bigcap_{n\geq 1}X_n(W_2^{(n)}) =\{0.3^\infty\},
$$
where $0.3^\infty=\sum_{k=1}^\infty A^{-k}\left(\begin{array}{c}3\\0\end{array}\right)=\left(\begin{array}{c}-3/2\\-3/14\end{array}\right)$ is a cut point of $T$.
 

\begin{figure}
\includegraphics[width=10cm,height=1.5cm]{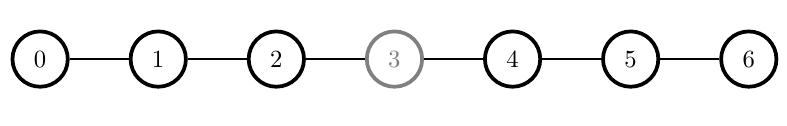}

\caption{First Hata Graph for Example~\ref{ExCNSA6B7}.}
\label{H1GraphA6B7}
\end{figure} 

\begin{figure}

\includegraphics[width=8.5cm,height=10cm]{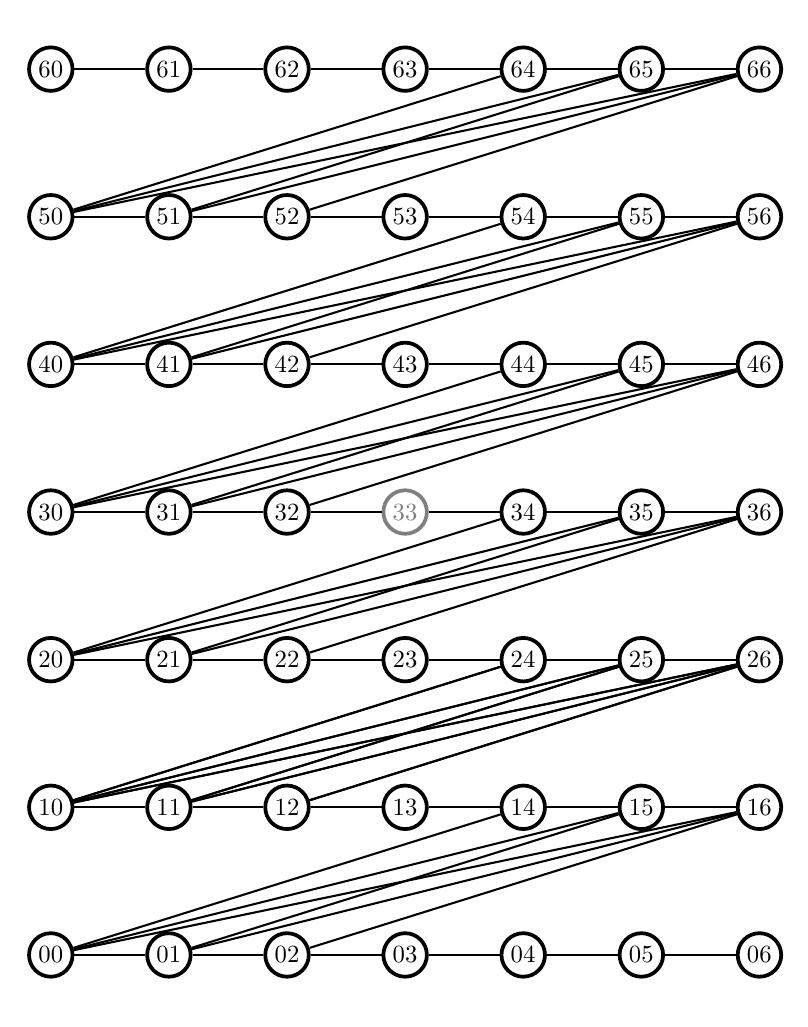}
\caption{Second Hata Graph for Example~\ref{ExCNSA6B7}.}
\label{H2GraphA6B7}
\end{figure}

\end{example}

\end{section}

\bibliographystyle{amsplain}

\end{document}